\documentclass[12pt]{amsart}
\usepackage{amscd,amssymb,subfigure}
\usepackage[all]{xy}
\usepackage{graphicx,verbatim}

\usepackage[colorlinks,plainpages,backref,urlcolor=blue]{hyperref}

\title [Line arrangements: classification and monodromy]{Pencil type line arrangements of low degree: classification and monodromy}
\date {21 May 2013}

\numberwithin{equation}{section}

\theoremstyle{plain}
\newtheorem{theorem}[subsection]{Theorem}
\newtheorem{lemma}[subsection]{Lemma}
\newtheorem{prop}[subsection]{Proposition}

\theoremstyle{definition}
\newtheorem{remark}[subsection]{Remark}
\newtheorem{definition}[subsection]{Definition}
\newtheorem{example}[subsection]{Example}

\newcommand{\A}{\mathcal{A}}

\newcommand{\LL}{{\mathcal L}}
\newcommand{\Ka}{{\mathcal K}}

\newcommand{\C}{\mathbb{C}}
\newcommand{\F}{\mathbb{F}}
\newcommand{\Z}{\mathbb{Z}}
\newcommand{\Q}{\mathbb{Q}}

\newcommand{\T}{\mathbb{T}}

\renewcommand{\k}{\Bbbk}

\DeclareMathOperator{\Hom}{Hom}

\begin{document}

\date{4 June 2013}

\author[A. Dimca]{Alexandru Dimca$^1$}
\address{Univ. Nice Sophia Antipolis, CNRS,  LJAD, UMR 7351, 06100 Nice, France.}
\email{dimca@unice.fr}

\author[D.~Ibadula]{Denis~Ibadula}
\address{Ovidius University, Faculty of Mathematics and Informatics, 124 Mamaia Blvd., 900527 Constan\c{t}ta, Romania}
\email{denis@univ-ovidius.ro}

\author[A.~D.~Macinic]{Daniela Anca~Macinic$^2$}
\address{Simion Stoilow Institute of Mathematics, 
P.O. Box 1-764,
RO-014700 Bucharest, Romania}
\email{Anca.Macinic@imar.ro}

\thanks{$^1$ Supported by Institut Universitaire de France} 

\thanks{$^2$ Supported by a grant of the Romanian National Authority for Scientific
Research, CNCS  UEFISCDI, project number PN-II-RU-PD-2011-3-0149} 



\begin{abstract}
The complete classification of $(3,3)$-nets and of $(3,4)$-nets with only double and triple points is given. Up to lattice isomorphism, there are exactly $3$ effective possibilities in each case, and some of these provide new examples of pencil-type line arrangements.
For arrangements consisting $\leq 14$ lines and having points of multiplicity $\leq 5$, we show that the
non-triviality of the monodromy on the first cohomology $H^1(F)$ of the associated Milnor fiber $F$ implies the arrangement is of reduced pencil-type.
In particular,  the monodromy is determined by the combinatorics in such cases.
\end{abstract}

\maketitle

\tableofcontents

\section{Introduction}
\label{intro}

Let $\mathcal{A}$ be  an essential arrangement  of $d$ lines in $\mathbb{P}^2\C$, given by $L_i=0, \; i=\overline{1,d}$, where $L_i$ 
are linear forms.  We denote by $M$ be the complement of the set of lines of $\A$ in $\mathbb{P}^2 \C$. Let $F$ be the the affine surface in $\mathbb{C}^3$  given by $Q=1$, with $Q=\Pi_{i}L_i$, the 
  Milnor fiber of the arrangement $\A$, that comes endowed with the monodromy action $h:F\rightarrow F$, $h(x)=exp(2\pi i/d)\cdot x$.

\medskip

It is an interesting open question whether the monodromy operator $h^1:H^1(F) \to H^1(F)$ is combinatorially determined, i.e. determined by the intersection lattice $L(\A)$. 

\medskip
Assume for the moment that the line arrangement $\A$ has only double and triple points. Then it is known that $h^1:H^1(F) \to H^1(F)$ is trivial unless $d=3m$ for some integer $m\geq 1$, and then only the eigenvalues $1$, $\epsilon= \exp(2\pi i/3)$ and $\epsilon ^2$ are possible, see for instance \cite{D2}, Cor. 6.14.15.

 Even in this special case, it is not known even whether the first Betti number $b_1(F)$ is determined by the combinatorics.
However, in a recent paper \cite{Libgober}, A. Libgober has followed the approach started in \cite{CL} and has shown that if $h^1 \ne Id$, then necessarily 
$\A$ is a (reduced) pencil-type arrangement, i.e. there is a pencil $aQ_1+bQ_2$ of curves of degree $m$ in $\mathbb{P}^2\C$
such that if we set $Q_3=Q_1+Q_2$, then $Q=Q_1\cdot Q_2\cdot Q_3$. We say also in such a case that $\A$ is a $(3,m)$-net. In the case $m=3$, an elementary proof of Libgober's result is given below in Proposition \ref{libg}, in order to prepare the reader for the more complex, but similar proof of Theorem \ref{withcvaduple}.

The main results of this paper are as follows.

In section 2 we give  the complete classification of $(3,3)$-nets, see Theorem \ref{3nets}, and of $(3,4)$-nets with only double and triple points, see Theorem \ref{3,4-nets}. Up to lattice isomorphism, there are exactly $3$ effective possibilities in each case, and some of these provide new examples of pencil-type line arrangements.

The proofs depend on the description given by Urz\'ua \cite{Urzua} (see also Stipins   \cite{St}) of the $(3,q)$-nets for $q = 3$ and $q = 4$.
In both these papers, the classification of nets considers only to the realisation of the multiple points given by the Latin squares, while here we take into account all multiple points of the corresponding arrangement.

In section 3, using a technique inspired from \cite{MP}, we show that for arrangements of $\leq 14$ lines in $\mathbb{P}^2\C$  with points of multiplicities  $\leq 5$, the non-triviality of $h^1$ implies that the arrangement is of reduced pencil-type, see Theorem \ref{withcvaduple}.
In particular, this shows that  the monodromy is determined by the combinatorics in all such cases.
This result is the best possible as far as only {\it reduced} pencils are considered: indeed, there is a line arrangement $\A$ with $|\A|=15$ and having only points of multiplicity $\leq 6$, where the non-triviality of $h^1$ comes from the existence of a {\it non-reduced} pencil, alias a multinet, see
Remark \ref{nonreduced}.

A final point on notation: $M$ denotes the complement of the line arrangement in this Introduction and in section 3, and a (Latin square) matrix in section 2.

\section{Classification of (3,3) and (3,4) nets}
\label{sect1}

 \begin{definition}
\label{mixt}
Let $\mathcal{A}=\mathcal{A}_1\cup \mathcal{A}_2\cup \mathcal{A}_3$ be a  $(3,q)$-net (see \cite[Def 1.1]{Y}). We call { \it mixed} triple points the points of $\A$ that appear as intersections of three lines, one in $\A_1$, one in $\A_2$, and one in $\A_3$. This is exactly the set of base points of the corresponding pencils, i.e. it is given by $Q_1=Q_2=0$.
\end{definition}

We need to recall the relation between $(3,q)$-nets and Latin squares.

A {\it Latin square} is a $q \times q$ matrix filled with $q$ different symbols, each occurring exactly once in each row and exactly once in each column. The symbols in our case are the numbers $\{1,2, \dots, q\}$. Latin squares are multiplication tables of finite quasigroups. 
If $\mathcal{A}=\mathcal{A}_1\cup \mathcal{A}_2\cup \mathcal{A}_3$ is a  $(3,q)$-net, then the $q^2$ mixed triple points are encoded  by  a $q \times q$ Latin square (see, for instance, \cite{St}). First choose an order for the lines on each of the subarrangements $\A_1, \A_2, \A_3$. Say, $\A_j=\{L^j_1, L^j_2, \dots, L^j_q \}, \; 1 \leq j \leq 3$. Then, the element at the intersection of the line $k$ and column $l$ in the corresponding Latin square is the label $\alpha (k,l) \in \{1,2, \dots, q \}$ of the line in $\A_3$ that passes through the intersection $L^1_k \cap L^2_l$, i.e. $L^1_k \cap L^2_l \cap
L^3_{\alpha(k,l)} \ne \emptyset$.

Since we are interested in the realisation of the arrangement $\A$ as a curve, the ordering on the lines inside a subarrangement, and respectively the ordering of the subarrangements, are not relevant. Accordingly, one may define an equivalence relation on the set of $q \times  q$ Latin squares. The equivalence class (referred further as {\it main class},  as in \cite{Urzua}) of a Latin square $M$ contains all the Latin squares obtained by rearranging the lines, columns or symbols of $M$ (this corresponds to reordering the lines inside the subarrangements $\A_i$), respectively by permuting the sets of lines, columns and symbols among them (this corresponds to reordering the labels $\{1,2,3\}$ of the subarrangements $\A_1, \;\A_2, \;\A_3$). An inventory of main classes of Latin squares for $q \leq 6$ is given in \cite[Sect.4]{Urzua}; it follows immediately that, up to lattice isomorphism, there is only one $(3,2)$-net. 

We follow the cases $q=3$ and $q=4$.  

 Let us consider first the case $q=3$, when there is a unique main class of Latin squares (see for instance \cite{K}), represented by:

\begin{equation}
\label{square3}
 M:= \begin{array}{ccc}
1 & 2 & 3 \\
3 & 1 & 2\\
2 & 3 & 1 \end{array}
\end{equation}

\begin{theorem} 
\label{3nets}
Let $\mathcal{A}=\mathcal{A}_1\cup \mathcal{A}_2\cup \mathcal{A}_3$ be an arrangement of $9$ lines in $\mathbb{P}^2 \C$ defined by a $(3,3)$-net. Then exactly one of the following situations holds:
\begin{enumerate}
	\item If $\mathcal{A}$ has only triple points (i.e. no double points and $12$ triple points: $9$ mixed triple points and one triple point in each subarrangement $\mathcal{A}_i$, $i=\overline{1,3}$), then $\mathcal{A}$ is lattice isomorphic to the Ceva arrangement.

	\item If $\mathcal{A}$ has $9$ double points (and thus  the only triple points are the $9$ mixed triple points), then $\mathcal{A}$ is lattice isomorphic to the  Hesse arrangement ($\mathcal{A}$ is the union of three singular fibers out of the four singular fibers of the Hesse pencil $a(x^3+y^3+z^3)+bxyz$). 
	
	\item If $\mathcal{A}$ has $6$ double points (and thus $10$ triple points: $9$ mixed triple points and a triple point in one of the subarrangements), then, up to a lattice isomorphism, the lines of $\mathcal{A}$ are given by: $L_1=(y)$, $L_2=(\frac{1}{b}x+y+z)$, $L_3=(\frac{b}{b-1}x+z)$, $L_4=(x)$, $L_5=(x+by+z)$, $L_6=(by+z)$, $L_7=(x+b(1-b)y)$, $L_8=(x+y+z)$, $L_9=(z), \; b\in \mathbb{C}\setminus \{0, 1, \epsilon\},\;\epsilon^3=-1$.
\end{enumerate}
\end{theorem}

\begin{proof}
The proof is based on the description of the realization space for  $(3,3)$-nets from  \cite{Urzua} (see also \cite{St}).  While this realization result takes into account just the pattern of the mixed triple points, we also look at the  multiple points inside each subarrangement $\A_i$.
Unless otherwise stated, from now on isomorphism
between two
line arrangements means lattice isomorphism.
 
If $\A$ has only triple points, since $(3,3)$-nets are defined by a single class of Latin squares, we have only one isomorphism class of arrangements, represented by the Ceva arrangement.

Likewise, if $\A$ has only double points apart from the mixed ones, then again we have only one isomorphism class of arrangements, represented by the Hesse  arrangement described in the theorem.

In \cite[Sect 4]{Urzua} Urz\'{u}a gives a set of equations for the lines of a realizable  $(3,3)$-net, with coefficients in a  two parameter space: $L_1=(y)$, $L_2=(\frac{1}{c}x+y+z)$, $L_3=(\frac{b}{b-1}x+z)$, $L_4=(x)$, $L_5=(x+cy+z)$, $L_6=(by+z)$, $L_7=(x+c(1-b)y)$, $L_8=(x+y+z)$, $L_9=(z)$ with  $b, c \in \C \setminus \{a\; finite\;number \;of\;elemements\}$ (including, for instance, $b, c \neq 0, 1$). A direct linear algebra computation using these equations shows, on one hand, that it is impossible to have arrangements $\A$ with $11$ triple points and, on the other hand, that $(3,3)$-nets with $10$ triple points do exist. 

It is not difficult to see that all $(3,3)$-nets with $10$ triple points are isomorphic, by taking into account the fact that if we permute the labels $\{1,2,3\}$ of the subarrangements the Latin square M from  \ref{square3} remains unchanged. 
\end{proof}

\begin{example} 
\label{new}
An example of $(3,3)$-net with $10$ triple points is the Pappus arrangement (the configuration described by the Pappus hexagon theorem, as in \cite[Fig. 6 (b)]{ACTY}). Equations for the lines of this arrangement are given in \cite[Table 5]{ACTY}. 

Note that the arrangement that realizes the classical configuration $(9_3)_1$ from \cite[Example 10.9]{Su} (having $9$ triple points), to which the author refers also as Pappus arrangement,  is a different arrangement (the above Pappus arrangement
can be seen as a degeneration of a family of $(9_3)_1$ arrangements, in which three double points collapse to yield a triple point).
\end{example}

\begin{remark}
\label{isotopy}
A sharper form of Theorem \ref{3nets} can be obtained using the classification of the realization spaces of  line arrangements of up to  $9$ lines in  \cite{NY} (see also \cite[Thm 3.9]{Ye}), since any two arrangements contained in the same connected component of the realization space are lattice isotopic.  
Recall that two arrangements are called \textit{lattice isotopic} if they are connected by a one-parameter family of arrangements with
constant intersection lattice.

The Ceva arrangement contains as a subarrangement a MacLane arrangement (described in \cite[Example 10.7]{Su}), hence any two realizations of Ceva are either in the same connected component of the realization space (hence lattice isotopic) or one of them is in the same connected component of the realization space as the conjugate of the other. 

$(3,3)$-nets  with precisely $9$ triple points  are actually lattice isotopic  (see for instance \cite[Prop. 3.7]{Ye} and \cite[Prop. 4.6]{ACTY}), since only one of the three possible combinatorial types of arrangements, described in \cite[Thm2.2.1]{Gr} (see also \cite[Fig. 5]{ACTY}), corresponds to a net, and its realization space is connected.

One can also see that $(3,3)$-nets with $10$  triple points are lattice isomorphic using \cite[Prop. 4.8]{ACTY}, since only one of the two combinatorial types given there can be a net. It follows that a $(3,3)$-net with 10 triple points is combinatorially equivalent to the Pappus arrangement. Hence $(3,3)$-nets  with $10$ triple points are  lattice isotopic, by  \cite[Thm 3.15]{NY}).
\end{remark}

Consider now the case of (3,4)-nets. Then there are two main classes of  Latin squares (see \cite{K}), with representatives:

\begin{equation}
\label{squares}
 M_1:= \begin{array}{cccc}
1 & 2 & 3 &4  \\
2 & 3 & 4 & 1\\
3 & 4 & 1 & 2\\
4 & 1 & 2 & 3 \end{array} \;\;\;\;\;\;\; M_2:=\begin{array}{cccc}
1 & 2 & 3 &4  \\
2 & 1 & 4 & 3\\
3 & 4 & 1 & 2\\
4 & 3 & 2 & 1 \end{array} 
\end{equation}

 For each one of them, we have a realization space (as arrangements of hyperplanes in $ \mathbb P^2 \C$)  of strictly positive dimension (\cite{Urzua}, \cite{St}). We give a classification of the $(3,4)$-nets having
only double and triple points. Since the Latin  squares give the intersection pattern of mixed triple points, it remains to test the existence of arrangements with only double and triple points inside each of the subarangements $\A_i$.

It is easy to observe that, when $\A$ is a reduced pencil,  there are two possible configurations  for the lines in the subarrangements $\A_j$: we say that $\A_j$ is of type $(i)$ if $\A_j$ has a triple point, and of type $(ii)$ otherwise.

\begin{theorem}
\label{3,4-nets}
Let $\mathcal{A}=\mathcal{A}_1\cup \mathcal{A}_2\cup \mathcal{A}_3$ be an arrangement of $12$ lines in $\mathbb{P}^2 \C$ defined by a $(3,4)$-net via a Latin square $M_j, \; j\in \{1,2\},$ and having only
double
and triple points. Then, for each $j$, exactly one of the following situations holds:
\begin{enumerate}
\item All the subarrangements $\A_i, \;i \in \{ 1,2,3\},$ are of type (ii), $\A$ having $16$ triple points.
\item All the subarrangements $\A_i, \;i \in \{ 1,2,3\},$ are of type (i), $\A$ having $19$ triple points.
\item Two of the subarrangements $\A_i$ are of type  (ii) and the third one is of type  (i), $\A$ having $17$ triple points.
\end{enumerate}
Moreover, the number of triple points classifies the arrangements up to isomorphism, that is, each of the above situations is represented by a unique lattice  isomorphism class of the arrangement $\A$.
\end{theorem}

\begin{proof}
A straightforward but lengthy computation, (even taking into account the symmetries)  partly done by MAPLE, using the description of the realisability space of $(3,4)$-nets given in \cite{Urzua},
shows both the existence of examples for each item in the list and disproves the existence of an arrangement $\A$ that has a triple point in two out of three subarrangements $\A_i, \; i \in \{1,2,3\}$. Obviously the arrangements described in $(1)$ are all lattice isomorphic.

Let us discuss now the second case. Assume $\A$ has a triple point in each of the subarrangements $\A_i, \in \{1,2,3\}$. In fact, a (partly MAPLE assisted) computation using the equations (with coefficients having three degrees of freedom) of the lines of a realizable $(3,4)$-net given in \cite[Sect.4]{Urzua} shows that, assuming there are triple points in two of the subarrangements $\A_i$, there must be a triple point in the third subarrangement. 
Consider first that the arrangement $\A$ is associated to the Latin square $M_1$.

Using the equations of the lines given in \cite{Urzua}, we have identified 16 possible configurations for the triple points of $\A$ (only the placement of non-mixed triple points may vary). 
For instance, if the triple point inside $\A_1$ is given by the intersection of the lines $(L_1, L_2, L_3)$, then 4 configurations are obtained by varying the lines  passing through the triple point in $\A_2$, as follows. If the triple point in $\A_2$ is the intersection of $(L_5, L_6, L_7)$, then one gets in $\A_3$ the triple point $(L_{10}, L_{11}, L_{12})$; if the triple point in $\A_2$ is the intersection of $(L_5, L_6, L_8)$, then one gets in $\A_3$ the triple point $(L_{9}, L_{10}, L_{11})$; if the triple point in $\A_2$ is the intersection of $(L_5, L_7, L_8)$, then one gets in $\A_3$ the triple point $(L_{9}, L_{10}, L_{12})$; respectively, if the triple point in $\A_2$ is the intersection of $(L_6, L_7, L_8)$, then one gets in $\A_3$ the triple point $(L_{9}, L_{11}, L_{12})$. The 12 remaining configurations appear when we vary the lines through the triple point in $\A_1$. We show that all 16 arrangements (configurations) $\A$ are lattice isomorphic. The first 4 arra
 ngements can be obtained one from another by re-labelling the lines inside the subarrangements $\A_2$ and $\A_3$ by a permutation by some power $\sigma^i, 1 \leq i \leq 3$, of the cycle $\sigma:=(1234)$, the same power for both $\A_2$ and $\A_3$. 

The same property holds for the other three sets of 4 configurations. So we are left, up to lattice isomorphism, to 4 configurations.  Since $M_1$ is symmetric in $\A_1, \; \A_2$ (i.e. replacing the lines by columns in the same order leaves the square unchanged), it follows easily that all configurations are isomorphic (just define lattice isomorphisms that permute the subarrangements $\A_1$ and $\A_2$). 

The $M_2$ case is treated similarly.

When $(3)$ happens, we notice that, as lattice isomorphism type, there is no distinction between an arrangement $\A$ with the non-mixed triple point in the arrangement $\A_i$ and an arrangement $\A$ with the non-mixed triple point in the arrangement $\A_j, j \neq i$. The proof of this claim uses as before the symmetries  that appear
 in the Latin square $M_i, \; i=1,2$, and an appropriate re-labelling  of the lines.  First we show that we may assume without losing the generality that the triple point is in $\A_1$.
Notice that both Latin squares $M_1$ and $M_2$ are symmetric in $\A_1$ and $A_2$. Moreover, the second Latin square is symmetric with respect to all the sub-arrangements  $\A_i$, so in this case one can assume that the triple point is, for instance, in $\A_1$.
In the $M_1$ case, we still need to find a lattice isomorphism between an arbitrary  arrangement $\A'=\A'_1 \cup \A'_2  \cup \A'_3$ with the triple point in $\A'_3$ and an arrangement with the triple point in $\A_1$. 

It is enough to re-label some of the lines in $\A'$ as follows: $L_i 	\leftrightarrow L_{i+8}$ for $i= 1, \dots, 4$ and $L_6 \leftrightarrow	 L_8$, to obtain an arrangement with a triple point in $\A_1$ and the pattern of mixed triple intersection points described by $M_1$ (that is, the representative of the class as presented in \ref{squares}, not some other Latin square in its main class). 
 
 It remains to be seen that any two arrangements with a  non-mixed triple point in $\A_1$ are isomorphic.
 This triple point may appear as the intersection of $(L_1, L_2, L_3)$, $(L_1, L_2, L_4), \; (L_1, L_3, L_4)$ or respectively $(L_2, L_3, L_4)$. So there are a priori $4$ types of arrangements $\A$ to consider. We show that all are lattice isomorphic to an arrangement where the triple point is given by the intersection of $(L_1, L_2, L_3)$. 
 
 Assume $\A$ has the mixed triple points given by $M_1$.
 For instance, the arrangement with the triple point $(L_1, L_2, L_4)$, via a re-label of its lines in $\A_1$ and $\A_3$ by the rule $L_i 	\leftrightarrow L_{\sigma(i)}, \; L_{i+8}\leftrightarrow L_{\sigma(i+8)}, i=1, \dots, 4$, gives an arrangement with the triple point $(L_1, L_2, L_3)$ and the pattern of mixed triple points unchanged.  Similarly, if the arrangement has the triple point $(L_1, L_3, L_4)$, (respectively $(L_2, L_3, L_4)$), the lines in $\A_1$ and $\A_3$ may be re-labelled  by the rule $L_i \leftrightarrow L_{\sigma^2(i)}, \; L_{i+8} \leftrightarrow	 L_{\sigma^2(i+8)}, i=1, \dots, 4$, (respectively $L_i \leftrightarrow L_{\sigma^3(i)}, \; L_{i+8} \leftrightarrow L_{\sigma^3(i+8)}, i=1, \dots, 4$), to give arrangements with the triple point $(L_1, L_2, L_3)$ and the pattern of mixed triple points given by the Latin square $M_1$.
 
If $\A$ has the mixed triple points given by $M_2$, a similar argument applies.

\end{proof}
 
 We give a series of examples of arrangements realizable over $\Q$ illustrating all the situations from the above theorem.

 \begin{example}
 \label{3,4}
 If $\A$ is associated to a Latin square  of type $M_1$:
 \begin{enumerate}
\item  The arrangement $L_1=(y)$, $L_2=(10x+y+z)$, $L_3=(10x+76y+43z)$, $L_4=(5x+z)$, $L_5=(x)$, $L_6=(40x+19y+40z)$, $L_7=(175x+10y+43z)$, $L_8=(2y+z)$, $L_9=(10x-y)$, $L_{10}=(x+y+z)$, $L_{11}=(175x+76y+43z)$, $L_{12}=(z)$ has no other triple point besides the $16$ mixed triple points. All intersection points in each subarrangement $\mathcal{A}_i$ are double. Thus, the arrangement has $6\cdot 3=18$ double points.

\item The arrangement $L_1=(y)$, $L_2=(-x+y+z)$, $L_3=(2x+4y+3z)$, $L_4=(-x+z)$, $L_5=(x)$, $L_6=(x+2y+z)$, $L_7=(-x+y+3z)$, $L_8=(2y+z)$, $L_9=(-x-y)$, $L_{10}=(x+y+z)$, $L_{11}=(-x+4y+z)$, $L_{12}=(z)$ has three triple points (besides the $16$ mixed triple points), one in each subarrangement $\mathcal{A}_i$: $L_1\cap L_2\cap L_4$, $L_5\cap L_6\cap L_8$, $L_9\cap L_{10}\cap L_{12}$. Thus, in this arrangement there are $9$ double points, $3$ double points in each subarrangement $\mathcal{A}_i$. 
 
 \item The arrangement $L_1=(y)$, $L_2=(-2x+y+z)$, $L_3=(-2x+4y+z)$, $L_4=(5x+z)$, $L_5=(x)$, $L_6=(8x-25y+8z)$, $L_7=(13x-2y+z)$, $L_8=(2y+z)$, $L_9=(10x+5y)$, $L_{10}=(x+y+z)$, $L_{11}=(13x+4y+z)$, $L_{12}=(z)$ has one triple point (besides the $16$ mixed triple points) in the subarrangement $\mathcal{A}_2$: $L_5\cap L_7\cap L_8$. All the other intersection points are double. Hence, there are $15$ double points. 
 
 \end{enumerate}
 
 If $\A$ is associated to a Latin square of type $M_2$:
 \begin{enumerate}
 
  \item The arrangement $L_1=(y)$, $L_2=(2x+y+z)$, $L_3=(12x+15y+13z)$, $L_4=(12x+z)$, $L_5=(x)$, $L_6=(4x+5y+4z)$, $L_7=(24x+12y+13z)$, $L_8=(3y+z)$, $L_9=(12x-3y)$, $L_{10}=(x+y+z)$, $L_{11}=(24x+15y+13z)$, $L_{12}=(z)$ has only double points ($3$ in each subarrangement $\mathcal{A}_i$) besides the $16$ mixed triple points. 
 
 \item The arrangement $L_1=(y)$, $L_2=(x+5y+5z)$, $L_3=(-2x-8y+z)$, $L_4=(2x+5z)$, $L_5=(x)$, $L_6=(x+4y+z)$, $L_7=(4x+10y-5z)$, $L_8=(-2y+z)$, $L_9=(2x+10y)$, $L_{10}=(x+y+z)$, $L_{11}=(4x+40y-5z)$, $L_{12}=(z)$ has three triple points besides the $16$ mixed triple points, one in each subarrangement $\mathcal{A}_i$: $L_1\cap L_2\cap L_4$, $L_5\cap L_7\cap L_8$, respectively $L_9\cap L_{10}\cap L_{11}$  and $9$ double points, $3$ double points in each subarrangement $\mathcal{A}_i$. 

\item The arrangement $L_1=(y)$, $L_2=(3x+y+z)$, $L_3=(3x-y+5z)$, $L_4=(3x+z)$, $L_5=(x)$, $L_6=(3x-y+3z)$, $L_7=(9x+3y+5z)$, $L_8=(-2y+z)$, $L_9=(3x+2y)$, $L_{10}=(x+y+z)$, $L_{11}=(9x-y+5z)$, $L_{12}=(z)$ has only one triple point $L_1\cap L_2\cap L_4$ besides the $16$ mixed triple points and $15$ double points ($3$ double points in the subarrangement $\mathcal{A}_1$ and $6$ double points in each of the subarrangements $\mathcal{A}_2$ and $\mathcal{A}_3$). 
 
 \end{enumerate}
 \end{example}

 \begin{remark}
 \label{3,5}
Although one can easily produce examples of  $(3,5)$-nets that have at most triple points, a complete classification by the lattice isomorphism type is work in progress by the authors.
\end{remark}

 \section{Nets and monodromy}
 \label{sect2}
 
When necessary, we may look at  $\A$ as an essential central arrangement in $\C^3$. This does not affect the definition of the associated Milnor fiber, which is our object of focus for this section.
In this context, we make a brief inventory of some useful definitions and results.

Let  $A_{\k}^*(\A)$ be the Orlik-Solomon algebra with coefficients over the field $\k$ of the arrangement $\A$. By definition, $A_{\k}^1(\A)$ is freely generated by $\{ a_H \}_{H\in \A}$. Let $\omega:= \sum_{H} a_H$, and denote
  by $\mu_{\omega}$ the multiplication by $\omega$ in $A_{\k}^*(\A)$. The {\it Aomoto complex} associated to $\omega$ is the cochain complex:

\begin{equation}
\label{eq:aok}
\big( A_{\k}^*(\A), \mu_{\omega} \big)= \{ A_{\k}^*(\A)\stackrel{\mu_{\omega}
}{\longrightarrow} 
A_{\k}^{*+1}(\A)\}_{*\ge 0}\, ,
\end{equation}

If $\k=\F_p$, set
\begin{equation}
\label{eq:modres}
\beta_{qp}(\A):= \dim_{\k} H^q (A_{\k}^*(\A), \mu_{\omega})\, \quad {\rm for}\quad 
q\ge0 \,.
\end{equation}
 the {\em Aomoto-Betti numbers}.
 
 It is well known that
we have 
an equivariant decomposition, consequence of the order $d$ geometric monodromy  of the Milnor fiber:
\begin{equation}
\label{eq:ciclo}
H_1(F_{\A}, \Q)= \bigoplus_{m |d} \big( \frac{\Q[t]}{\Phi_m} \big)^{b_{m}(\A)},
\end{equation}
where $\Phi_m$ is the $m$th cyclotomic polynomial and the exponents $b_{m}(\A)$ depend on $m$ and $\A$; see for instance
\cite{OT, L}.  When $\A$ has only double and triple points, then $b_m(\A) \ne 0$ implies
that either $m=1$ or $m=3$. One has $b_{m}(\A)=0$ for $m>1$ when $m$ does not divide 
$d=|\A|$ or if there are no points in $\A$ of multiplicity a multiple of $m$. In particular $h^1=Id$ when $d$ is a prime number, see \cite{CS}.

The exponents $b_{m}(\A)$ are connected by {\it modular inequalities} to the Aomoto-Betti numbers, via local coefficients cohomology of the complement. 
To state them, let $\T(M)=\Hom (\pi_1(M), \C^*)$ be the affine torus parametrizing the rank one local systems on the hyperplane complement $M$ of $\A$. When $m|d$, with $d=|\A|$, we denote by $\rho (m)={\bf 1}/{m} \in \T(M)$ the rank one local system whose monodromy about each line $L \in \A$  is $\lambda (m)=\exp (2\pi i/m)$.

Recall the following inequality, playing a key role in the proofs below.

\begin{theorem}[\cite{CO, PS}]
\label{thm:cohen}
Assume $M$ is the complement of a central arrangement $\A$ and $\rho={\bf 1}/{p^s}$  a rational equimonodromical local system on $M$ with $p$ prime and $s \geq 1$, and denote $b_1(M, {\bf 1}/{p^s}):=\dim H^1(M, _{\rho}\C)$. Then
\[
b_1(M, {\bf 1}/{p^s}) \le \beta_{1p}(\A).
\]
\end{theorem}

 \noindent  On the other hand one knows that $b_m(\A)=b_1(M, {\bf 1}/m)$, hence  $b_{m=p^s}(\A)$ has $\beta_{1p}(\A)$ as upper bound.

 \noindent  If $X \in L(\A)$ is an arbitrary element in the intersection lattice of an arbitrary arrangement $\A$, denote by $m_X$ the number of hyperplanes that contain $X$, that is $m_X= \# \{H \in \A|\; X \subseteq H\}$ and let $\A_X$ be the subarrangement of $\A$ consisting of all hyperplanes that contain $X$.
 
The next lemma reduces the computation of $\beta_{1p}(\A)$ to solving a system of linear equations.
 
\begin{lemma} (\cite[Lemma 3.3]{LY})
\label{beta}
Let $\mathcal{A}$ be an arbitrary central arrangement and $p$ a prime. $\eta =\sum_{H\in \A}\eta_H a_H \in A_{\F_p}^{1}
(\mathcal{A})$ is a $1$-cocycle for \eqref{eq:aok} if and only if one has 
\begin{equation}
\label{div}
\sum_{H \supset X} \eta_{H}=0, \;\;if \;\;p \mid m_{X},
\end{equation}
or 
\begin{equation}
\label{ndiv}
\eta_{H}= \eta_{K},\;\;\forall\; H\ne K \;\;\in \mathcal{A}_{X}, \;\;if\; p\;\; \nmid m_{X}\,,
\end{equation}
for every rank $2$ element 
$X \in \mathcal{L}(\mathcal{A})$.
\end{lemma}

By the above Lemma \ref{beta}, the computation of $\beta_{1p}(\A)$ resumes to solving a system $\bf S$ with $\F_p$ coefficients of linear equations, with variables labelled by the lines of $\A$. The equations are in one to one correspondence to the multiple points of $\A$.
A solution for $\bf S$ is  precisely the set of coefficients in $\F_p$ of an arbitrary $1$-cocycle $\eta$ of the complex \eqref{eq:aok}. Hence Lemma \ref{beta} helps us compute the dimension of the space of $1$-cocycles of the complex \eqref{eq:aok} (that is, the dimension of the space of solutions of $\bf S$). It is easy to see that the dimension of the space of $1$-coboundaries of the same complex \eqref{eq:aok} is $1$.

 We call a solution $(a_H)_{H \in \A}$ of $\bf S$ {\it constant} if there exists $a \in \F_p$ such that $a_H=a$, for all $H\in \A$, and {\it non-constant} otherwise.  We will call $a_H \in \F_p$ {\it the weight} associated to the line $H$.

\begin{prop} \label{libg}
Let $\A \subset \mathbb{P}^2\C$ be an arrangement with $|\mathcal{A}| \leq 9$ such that $\A$ has only double and triple points and the monodromy operator $h^1:H^1(F)\rightarrow H^1(F)$ is not trivial. Then the arrangement $\mathcal{A}$ is composed of a reduced pencil. 
\end{prop}

\begin{proof} 
The above discussion shows that $h^1 \ne Id$ implies that $d=|\A|$ is
divisible by 3, i.e. $d=3$, or $d=6$, or $d=9$.
We give the details only for the case $d=9$ since the other two cases are
very simple and similar.

We make a discussion on the number of double points of the arrangement. 

Assume $\A=\{L_1, \dots, L_9\}$ has no double points (hence $\A$ has $12$ triple points). The rigidity of this configuration will lead us to the conclusion that $\A$ is isomorphic to the Ceva arrangement.  Since all the hyperplanes (lines) intersect each other, then each line contains $4$ triple points. One may assume that $L_1 \cap L_2 \cap L_3, \;L_1 \cap L_4 \cap L_7, L_1 \cap L_5 \cap L_8, \;L_1 \cap L_6 \cap L_9$ are the triple points on $L_1$. 

Let us look now at the triple points on $L_2$.  $L_2 \cap L_7$ must be a triple point, and the third line that contains this point may be one of the following: $L_5, L_6, L_8, L_9$. Without loosing generality, one may assume that $L_2 \cap L_5\cap L_7$ is the triple point. 
 We search the third line passing through the triple point  $L_2 \cap L_8$. It can be either $L_4$ or $L_6$ (or $L_9$, but this is reducible to the $L_6$ case, modulo a re-labelling $6 \leftrightarrow 9$).  If $L_2 \cap L_8$ would be on the line $ L_4$, then this would force $L_2 \cap L_6$ to be on the line   $L_9$, contradiction to the fact that $L_1 \cap L_6 \cap L_9$ is a triple point. Then necessarily  $L_2 \cap L_6\cap L_8$ is a triple point, and this forces the existence of the triple point $L_2 \cap L_4 \cap L_9$. Now, the triple point $L_4 \cap L_8$ can only be on the line $L_3$. Successively, we conclude that the line $L_9$ must pass through the triple point $L_7 \cap L_8$, then $L_6$ must pass through the triple point $L_3\cap L_7$. Finally, we necessarily must have the triple points $L_3 \cap L_5\cap L_9$ and $L_4 \cap L_5\cap L_6$. But this describes exactly the lattice of the Ceva arrangement.

For the remaining case, when $\A$ has at least one double point, we need to recall the inequality from \ref{thm:cohen}. By hypothesis, $b_{3}(\A)=b_1(M, {\bf 1}/3) >0$, hence by \ref{thm:cohen} one has $\beta_{13}(\A) \geq 1$.

The computation of $\beta_{13}(\A)$ resumes to solving over $\F_3$ the system $\bf S$ of linear equations, with variables labelled by the lines of $\A$. Since $\beta_{13}(\A) \geq 1$, the dimension of the space of solutions of the system  $\bf S$ associated to $\A$ is at least 2, so $\bf S$ admits a non-constant solution $(a_{L})_{L \in \A}$. We show that this implies that $\A$ is a reduced pencil.
The converse of this claim, i.e. the fact that a 
line arrangement defined by a $(3,3)$-net (that is, composed of a reduced pencil, by the terminology of  \cite{Dimca}) has a non-trivial $h^1$ is a known result also (see \cite[Thm 3.1]{DP}). This implies that  $b_{3}(\A)=b_1(M, {\bf 1}/3) >0$, hence  a reduced pencil necessarily has $\beta_{13}(\A) \geq 1$.

Assume $\A$ has at least one double point. The number of double points is divisible by three, so, in this case, $\A$ must have at least three double points. Fix three arbitrary double points in the lattice of $\A$. We consider two different cases. 

{\it (1)} If the three double points are not collinear, let $L_1, L_2, L_3$ be the triangle that realizes these points, and denote by $a_1, a_2, a_3$ the weights (for the considered non-constant solution) of  $L_1, L_2, L_3$. From \ref{ndiv} we get $a:=a_1=a_2=a_3$. We can choose a line $L_4$  with associated weight $b_1$, and $b_1 \neq a$. If $L_4$ meets any of the lines $L_1, L_2, L_3$ in a double point, then, by \ref{ndiv}, we get $b_1=a$. Otherwise,  $L_4$ must intersect each of the lines $L_1, L_2, L_3$ in triple points. So, there are three new lines $L_7, L_8, L_9$, with corresponding weights $c_1, c_2, c_3$ such that $(L_1, L_4, L_7), \; (L_2, L_4, L_9)$ and $(L_3, L_4, L_8)$ intersect in triple points. By \ref{div} we have that $c:= c_1=c_2= c_3$ and $a+b+c=0$. If any of the intersections $L_1 \cap L_8$ and $L_1 \cap L_9$ would be a double point then $a=b=c$. Otherwise, all intersections corresponding to couples of weights $(a_i, c_j)$ must be triple points.
Then through the intersection $L_1 \cap L_8$ passes another line, which must be different from $L_4$. Denote this new line by $L_5$. By \ref{div}, the weight $b_2$ corresponding to $L_5$ satisfies the equation $a_1+b_2+c_2=0$, hence $b:=b_1=b_2$. By a similar argument the intersection $L_1 \cap L_9$ contains a line $L_6$ different from $L_4$ and $L_5$. If $L_6$ would coincide with any of the lines of weights $a$ or $c$ then we would get $a=b=c$. Otherwise, $L_6$ must be different from any $L_i, \; i \neq 6$, and have corresponding weight $b_3=b$.

So $\A=\{L_1, \dots, L_9\}$ and we have up to now the following triple points: $L_1 \cap L_4 \cap L_7, \; L_2 \cap L_4 \cap L_9, \; L_3 \cap L_4 \cap L_8, \; L_1 \cap L_6 \cap L_9, \; L_1 \cap L_5 \cap L_8$. If any of the intersections of lines corresponding to couples of weights $(a_i, b_j), \;(a_i, c_j)$ or $(b_i, c_j)$, for $i,j \in \{1,2,3\}$ would be double points, then we would have $a=b=c$. Otherwise, we have the triple points $L_3 \cap L_5 \cap L_9, \;L_3 \cap L_6 \cap L_7, \;L_2 \cap L_6 \cap L_8, \;L_2 \cap L_5 \cap L_7$. For instance, let us explain in detail why should the point $L_3 \cap L_5 \cap L_9$ exist. We know that $L_5 \cap L_9$ cannot be a double point, since this would imply  $b=c=a$. Moreover, the third line that passes through this point must a line of weight $a$, otherwise we obtain once again $b=c=a$. In conclusion, $L_5 \cap L_9$ is a triple point, and the third line that contains this point must be $L_1, \;L_2$ or $L_3$. On the other 
 hand, the existence of the triple points $L_2 \cap L_4 \cap L_9$ and $L_1 \cap L_6 \cap L_9$ forces this  line to be $L_3$, hence the triple point $L_3 \cap L_5 \cap L_9$.

The existence of the other three triple points ($L_3 \cap L_6 \cap L_7, \;L_2 \cap L_6 \cap L_8, \;L_2 \cap L_5 \cap L_7$) bears a similar argument.

In conclusion, in this case, in order to have the space of solutions of $\bf S$ of dimension at least $2$, $\A$ needs to be a $(3,3)$-net, with $\A_1=\{L_1, L_2, L_3\}, \;\A_2=\{L_4, L_5, L_6\}, \;\A_3=\{L_7, L_8, L_9\}$. We already assumed that $\A_1$ contains only double points. As for the remaining subarrangements, it follows from Thm. \ref{3nets} one can have either double points in both $\A_2$ and $\A_3$, or double points in one of them and a triple point in the remaining subarrangement.

{\it (2)} Assume $\A$ has three collinear double points, say on a line $L_1 \in \A$. Denote by $L_2, L_3, L_4$ the lines that realize these double points by intersecting $L_1$. Then, since all the other lines in $\A$ intersect $L_1$ and $|\A|=9$, there must be another line $L_5$ that intersects $L_1$ in a double point. By \ref{ndiv}, the weights $a_1, \dots, a_5$ corresponding to the lines $L_1, \dots, L_5$ are equal. Denote by $a$ their common value. Take an arbitrary line $L_6 \in \A$, different from the previous ones, of weight $b$. If $L_6$ intersects any of the lines $L_1, \dots, L_5$ in a double point, then $a=b$. Otherwise, assume that $L_6$ intersects any of the lines $L_1, \dots, L_5$ in triple points. A triple point of type $L_i \cap L_j \cap L_6$, for $i,j \in \{1, \dots, 5\}$ would lead, by \ref{div}, to $a=b$. So the only possibility for $\bf S$ to have a solution space of dimension at least $2$ would be that $L_6$ to intersect the lines $L_1, \dots L,
 _5$ in triple points of type $K_i \cap L_i \cap L_6, \; i \in \{1, \dots, 5\}$, with $\{L_1, \dots, L_6\} \cap \{K_1, \dots, K_5\} = \emptyset$, but this contradicts the fact that $|\A|=9$.
 \end{proof}
 
 The above proof  gives an elementary argument for a result of Libgober (\cite[Thm 1.2]{Libgober}), in the case $m=3$ (where $d=3m$ is the number of lines of $\mathcal{A}$).  The result states that projective line arrangements with only double and triple points have non-trivial monodromy action on the degree $1$ cohomology module of the Milnor fiber must be reduced pencils.

In the final part of this paper we extend Libgober's result to line arrangements $\A$ with $|\A|\leq 14$, having points of multiplicity $\leq 5$.

 We give now a number of results to be used in the proof of Theorem \ref{withcvaduple}.
Unless otherwise specified, $\A$ is an arrangement  with points of multiplicity up to $5$, having at least a quadruple or a quintuple point (otherwise the result follows from Libgober's Theorem). 

\begin{lemma}
\label{F2partition}
Let $\A$ be such that $|\A|=12$ and $b_2(\A) \neq 0$. Then the system $\bf S$ with $\F_2$ coefficients admits a non-constant solution $(a_{H})_{H \in \A} \in \F_2^{|\A|}$, and there is a partition $\A_a \sqcup \A_b$ of $\A, \; |\A_a|=|\A_b|=6$, such that  $a_{H}=a, $ for all $H$ in $\A_a $, and $a_{H}=b \neq a$, for all $H$ in  $\A_b$.
\end{lemma}

\begin{proof}
The hypothesis $b_2(\A) \neq 0$ implies (via Theorem \ref{thm:cohen}) that $\beta_{12}>0$, that is, the space of solutions of $\bf S$ has dimension at least $2$. This means, we have a non-constant solution $(a_{H_{i}})_{H_i \in \A}$. This solution gives a partition of the set  $\overline {1, 12}$ into  $\{ i \in \overline {1, 12} \;|\; a_{H_i}=a \}\sqcup \{ i \in \overline {1, 12} \;| \; b_{H_i}=b \}, \; b\neq a$, and consequently a partition $\A_a \sqcup \A_b$ of the set of lines of $\A$. We will prove that $|\A_a|=|\A_b|=6$.

 Consider the multiple points on a line $H \in \A_a$. To have $a \neq b$, each intersection point of $H$ to a line in $\A_b$ must be a quadruple point, containing two lines from $\A_a$ and two lines from $\A_b$. Hence the lines from $\A_b$ intersect $H$ in pairs, so $|\A_b|=2k, \;|\A_a| \geq k+1$ and $3k+1 \leq 12$. The only possible values for $k$ are $1, 2$ and $3$. It is easy to see that in the first two cases any line in $\A_b$ would contain a multiple point for which the associated equation  (\eqref{div} or \eqref{ndiv}) would translate into $a=b$.
\end{proof}

\begin{lemma}
\label{3quadruple}
If $|\A|=12$ and $\bf S$ admits a non-constant solution, then each line of the arrangement contains exactly $3$ quadruple points. 
\end{lemma}

\begin{proof}
In the notations of Lemma \ref{F2partition},  consider the intersections of a given line $L$  in $\A$ of weight, say, $a$, by an arbitrary line of weight $b \neq a$. Unless this intersection point  is a quadruple point formed by lines of weights $a, a, b, b$ we get $a=b$.  This means the lines of weight $b$ intersect $L$ in pairs, in quadruple points. Since there are six lines of weight $b$, there must be exactly three quadruple points on $L$, each giving (by \eqref{ndiv}) an equation of type $a+a+b+b=0$.  
\end{proof}

Let $\LL \sqcup \Ka$ be a partition of the set of lines of an arrangement $\A$. We name the multiple points of $\A$ by the induced partition of their lines.  For instance, a quadruple point in $\A$ is called of type $(LKLK)$ if it is the intersection of two lines in $\LL$ and two lines in $\Ka$.

\begin{lemma}
\label{14partition}
In the above notations, there is no arrangement $\A$ of $14$ lines that admits a partition into two subsets $\A=\LL \sqcup  \Ka$ such that $ |\LL|=6, \; | \Ka|=8$ and each line in $ \mathcal{L}$ contains exactly $4$ quadruple points of type $(LKLK)$, while each line in $  \Ka$ contains exactly $3$ quadruple points of type $(LKLK)$.
\end{lemma}

\begin{proof}
We find equations for the lines in $\LL$, depending in all of 4 parameters. Each line of type $K$, containing 3 quadruple points of type $(LKLK)$ (q.p. for short) will impose an equation, so we'll get 8 equations, giving in the end 4 distinct equations. Then we show that this system has no solution satisfying the imposed conditions.

{\it Step 1. Partition of $\LL$ into $3$ sets.}

Consider a line $L_1 \in \LL$. The remaining 5 lines have to determine 4 q.p. on $L_1$, so there is one of them, call it $L_2$, such that $A=L_1 \cap L_2$ is a double point (d.p. for short) of the arrangement.
Denote $L_3$ any of the remaining 4 lines in $\LL$. Each of the lines $L_1$ and $L_2$ intersect $L_3$ in q.p. points, so the remaining (unlabelled) 3 lines produce 2 q.p. on $L_3$. Let $L_4$ denote the line among them such that $B=L_3 \cap L_4$ is a d.p.. Denote the remaining lines by $L_5$ and $L_6$ and note that $C=L_5 \cap L_6$ is a double point.

{\it Step 2. The points $A,B,C$ are collinear}

To prove this we use Pascal Hexagon Theorem: if the vertices of a hexagon sit on a conic, then the intersection of the opposite edges are collinear points.

Our hexagon is the union of the 6 lines in $\LL$. The pairs of opposite edges are $(L_1,L_2)$,  $(L_3,L_4)$ and $(L_5,L_6)$. Choose the order $L_1,L_3,L_5,L_2,L_4,L_6$ (this does not restrict the generality, see step 4, where all the possible orderings are considered). Then the vertices are $v_1=L_1 \cap L_3$, $v_2=L_3 \cap L_5$, $v_3=L_5 \cap L_2$, $v_4=L_2 \cap L_4$, $v_5=L_4 \cap L_6$ and $v_6=L_6 \cap L_1$.

Consider the vertices $v_1$ and $v_3$. Note that there are 2 lines of type K passing through $v_1$.
Any such line will meet again the union of lines in $\LL$ exactly in two  q.p. constructed above, not situated on the lines $L_1$ or $L_3$. There are only two points of this type on $L_5$, so one of these lines is the line determined by $v_1$ and $v_3$. Call it $K_1$. We claim that the third q.p. on $K_1$ is exactly $v_5$. Indeed, the 3 q.p. on $K_1$ should involve all the 6 lines in $\LL$, each occurring exactly once, and this yields our claim.

In exactly the same way we show that there is a line $K_2$ containing the other 3 vertices $v_2,\;v_4$ and $v_6$. The union $K_1 \cup K_2$ is the conic allowing us to apply Pascal's Theorem.

{\it Step 3. The equations for the lines in $\LL$}

By throwing the point $v_6$ at infinity, and choosing well the coordinates $(x,y)$ in the affine plane $\C^2$ we can assume the following.

$A=(0,0)$, $B=(b,0)$ for $b \in \C^*$, $b \ne 1$ and $C=(1,0)$.

$L_1: x=0$, $L_2: y=x$, $L_3: y=a(x-b)$ and $L_4:y=c(x-b)$ with $a,c \in \C^*$ and $a \ne c$, $L_5: y=d(x-1)$ with $d \in \C^* \setminus \{1\}, d \neq a \neq c$ and $L_6: x=1$.

Hence the 4 parameters are $a,b,c,d$.

{\it Step 4. The $4$ equations for the $4$ parameters}

They are obtained as follows. We have to list the partitions of the set $\LL$ into $3 $subsets with cardinal two each, such that $L_1,L_2$, resp. $L_3,L_4$ and $L_5,L_6$ are not in the same subset. Here is the list, obtaining by considering all the possible circular ordering of the $6$ lines in $\LL$ such that $L_1,L_2$, $L_3,L_4$ and $L_5,L_6$ are  opposite edges.

1. $(L_1,L_3), (L_5,L_2), (L_4,L_6)$     2.$(L_1,L_3), (L_6,L_2), (L_4,L_5)$

3. $(L_1,L_4), (L_5,L_2),(L_3,L_6)$     4. $(L_1,L_4), (L_6,L_2), (L_3,L_5)$

5. $(L_1,L_5), (L_3,L_2), (L_6,L_4)$       6. $(L_1,L_5), (L_4,L_2), (L_6,L_3)$

7. $(L_1,L_6), (L_3,L_2),(L_5,L_4)$       8. $(L_1,L_6),(L_4,L2_),(L_5,L_3).$

Since there are $8 \;K$-lines and the $3$ q.p. on each such line give a partition of the lines in $\LL$ as described before, each partition corresponds to 3 q.p. which should be on a $K$-line. 
However, due to the converse of the Pascal's Hexagon Theorem, the 8 triplets give rise to only 4 distinct equations.

Indeed, if we write that the 3 points corresponding to the first partition $v1,v3,v5$ are collinear , then we get by Pascal's Theorem that the 3 points corresponding to the 8-th partition (which are nothing else but the vertices $v2,v4,v6$ in Step 2) are also collinear.

A simple analysis shows that the 4 independent equations come from the first 4 partitions above.
The corresponding equations are the following.
\begin{equation} \label{e1}
ab-bdc+dc-d=0.
\end{equation} 
\begin{equation} \label{e2}
(abc-cd)(1-b)-bc+d=0.
\end{equation} 
\begin{equation} \label{e3}
ad-abd+bc-d=0.
\end{equation} 
\begin{equation} \label{e4}
abd-ad-ab^2c-ab+d+abc=0.
\end{equation} 
It is easy to check that the system given by the equations \eqref{e1} , \eqref{e2} , \eqref{e3} and \eqref{e4} has no solution.
\end{proof}

 \begin{remark}
 \label{nonreduced}
  
In  \cite[Remark 1.3]{Dimca}, the first author gives an example of an arrangement consisting of $15$ lines having points of multiplicity $6$, which satisfies $h^1 \ne Id$, but it is not composed of a reduced pencil. This shows that the next result is optimal.

 \end{remark}

\begin{theorem}
\label{withcvaduple}
Let $\A \subset \mathbb{P}^2\C$ be an essential line arrangement with $|\mathcal{A}|\leq 14$   such that $\mathcal{A}$ has points of multiplicity up to $5$. Assume the monodromy operator $h^1:H^1(F)\rightarrow H^1(F)$ is not trivial. Then $\A$ is a either a reduced $(3,q \leq 4)$-net or a reduced $(4,3)$-net. 
In particular,  the non-triviality of the monodromy $h^1:H^1(F)\rightarrow H^1(F)$  is detected by the combinatorics for such line arrangements $\A$.
\end{theorem} 

\begin{proof}
We give details only for the cases $|\A|=12$ and $|\A|=14$. The other cases that may need a proof ($|\A| \in \{6, \;8, \; 9,\; 10\}$) may be treated in a similar manner, but are much simpler to analyse.
These two cases we consider are very different: in the case $|\A|=12$ we get $(3,4)$ or $(4,3)$ nets, while the case $|\A|=14$ is shown to be impossible.

Assume $|\A|=12$. Notice that, since we are dealing only with points of multiplicity up to $5$, by \cite[Thm.3.13]{MP} we have that $b_m(\A)=0, \forall m \neq 2,3,4$. Then the non-triviality of the monodromy
 implies that one of the exponents $b_2(\A), b_3 (\A), b_4(\A),$
  in the formula \eqref{eq:ciclo} is non-zero. By Theorem \ref{thm:cohen}, $b_2(\A)$ and $b_4(\A)$ have $\beta_{12}$ as upper bound, while $b_3(\A)$ has $\beta_{13}$ as upper bound. We will show that $\beta_{12} > 0$ implies $\A$ is a $(4,3)$-net,  while $\beta_{13} > 0$ implies $\A$ is a $(3,4)$-net.

For the rest of the proof we use the same method and notations as in second part of the proof of Proposition \ref{libg}, that is we use the key Lemma \ref{beta}.  As explained before, the computation of  $\beta_{1p}(\A), \; p=2,3$, comes down to solving a system $\bf S$ of linear equations over $\F_p$, with variables in one to one correspondence with the lines of $\A$ and one equation for each multiple point in $\A$ (the equations are described by Lemma \ref{beta}). The dimension of the space of solutions for $\bf S$ is equal to $\beta_{1p}+1$.

There are two cases to consider.\\
{\it Case (1)} $b_3(\A) >0$. Then $\beta_{13}>0$, hence the system $\bf S$ with $\F_3$ coefficients admits a non-constant solution $(a_H)_{H \in \A}$.

If $\A$ has only points of multiplicity $2$ and $3$, the claim follows from \cite[Thm 1.2]{Libgober}. 

Otherwise, we may assume that there is a point of multiplicity $4$  or $5$.

{\it Case (1.1)} Assume there are $4$ lines $L_1$, $L_2$, $L_3$ and $L_4$ in $\A$ that intersect in a quadruple point. If we denote $a_{L_i}:=a_i, \; i \in \overline{1,4}$, the corresponding weights, from \eqref{ndiv} we get $a:=a_1=a_2=a_3=a_4$. Since the chosen solution is non-constant, one can choose a new line $K_1$ of weight $b_1, \; b_1 \neq a$. 

If $K_1$  meets any of the lines $L_1, L_2, L_3$ or $L_4$ in a double, quadruple or quintuple point, then, by \eqref{ndiv}, we would get $b_1=a$, which is impossible. It follows that the line $K_1$ must intersect each of the lines $L_1, L_2, L_3$ and $L_4$ in triple points. But through a such triple point could not pass two lines with the corresponding weight equal to $a$ or to $b_1$ (because $a+a+b_1=0$ or $a+b_1+b_1=0$ would imply $b_1=a$). Hence, there exist other four lines $T_1, T_{2}, T_{3}$ and $T_{4}$ of weights $c_1, c_2, c_3$ respectively $c_4$ such that $(L_i, K_1, T_i)$, intersect in triple points for all $i \in \overline{1, 4}$.

Note that the lines $T_1, T_{2}, T_{3}$ and $T_{4}$ are four different lines because otherwise one of them would meet $K_1$ in two distinct points. By \eqref{div} we have that $a_i+b_1+c_i=0$, for all $i \in \overline{1,4}$, and thus $c_1=c_2=c_3=c_4=-a-b_1=:c$. Note also that $c\neq a$ and $c\neq b_1$.

Let us look next at the intersection point between $L_1$ and $T_2$. Since $a\neq c$, this intersection point does not have multiplicity $2,4$ or $5$. It will be then a triple point. But though this intersection point could not pass any $L_i$ (otherwise, by \ref{div}, $a+c+a=0$ and thus $a=c$, impossible) or $K_1$ (because $L_1$ could not meet $K_1$ in two distinct points). It follows then that through the intersection point between $L_1$ and $T_2$ passes a new line $K_2$ which should obviously have weight $b_2=b_1=:b$. 

Analogously, one can prove that $(L_1, T_3, K_3)$ and $(L_1, T_4, K_4)$ are triple points, where $K_3$ and $K_4$ are new lines which should have weights $b_3=b_4=b$. Moreover, the lines $K_1, K_2, K_3$ and $K_4$ are distinct (otherwise, one of them would intersect $L_1$ in two different points).

Hence, we have obtained a partition of  the arrangement $\A$ into $4$ subarrangements $\A_i, \; i = \overline{1,4}$, with $\A_1$ composed of the lines $L_1, L_2, L_3$ and $L_4$,  $\A_2$ composed of the lines $K_i, i= \overline{1,4}$ and $\A_3$ composed of the lines $T_i, i= \overline{1,4}$. Inside each of these three subarrangements one may  have double, triple or quadruple intersection points. Moreover,  through the intersection point between an arbitrary  line from $\A_i$ and an arbitrary line from $\A_j$, for $i,j \in \overline{1,4}, \; i \neq j$ must pass exactly one line from the third subarrangement $\A_k, \; k\in \overline{1,4}, \;k \neq i \neq j$, otherwise we would get $a=b=c$. But this is just the description of a $(3,4)$-net.

{\it Case (1.2)} The arrangement $\A$ does not have quadruple points, hence $\A$ contains at least one quintuple point. The lines involved in that quintuple point are of equal weight $a$. Assume we have a line of weight $b \neq a$. This line must intersect the $a$-lines in triple points, so there are $5$ lines of weight $c$ such that $a+b+c=0$, and another line of weight $b$. There are however intersection points of a line of weight $a$ with a line of weight $c$  not contained in any of the two lines of weight $b$, so we get $a=c$, and then $a=b=c$, contradiction.

{\it Case (2)} $b_3(\A) =0$. Then necessarily $b_2(\A)>0$ or $b_4(\A)>0$, any of those inequalities implying  $\beta_{12} >0$. Hence the system $\bf S$ with $\F_2$ coefficients admits a non-constant solution $(a_H)_{H \in \A} \in \F_2^{12}$.

By Lemmas \ref{F2partition}, \ref{3quadruple}, six of the weights $(a_H)_{H \in \A}$ are equal to some $a \in \F_2$ and the other six are equal to $b \neq a$ and each line contains exactly three quadruple points; there are no quintuple points.

 \noindent To simplify the notation, we identify in what follows the lines and their weights. 
 
\noindent Take an arbitrary line of weight, say, $a_1=a$. As seen before, $a_1$ contains three quadruple points, identified to quadruplets of weights, say, $(a_1, a_2, b_1, b_2), (a_1, a_3, b_3, b_4)$ \\
and $(a_1, a_4, b_5,b_6)$ with $a_i=a$ and   $b_i=b$. The line $a_1$ intersects two more lines $a_5=a$ and $a_6=a$  into either a triple or two double points.

\noindent This would suggest a partition of $\A$ into subarrangements, as such: $\A_1=\{a_1, a_5,  a_6\}$, $\A_2=\{ a_2,a_3,  a_4\}$, and two other subarrangements each containing three of the six lines of type $b$ that are  apparent when considering the multiple points on a $b$-line.

To prove that this partition defines a net structure on $\A$ it is enough to check that, if $( a_1,   a_5, a_6)$ is not a triple point, then $(a_5, a _6)$ is a double point; then the same would apply for $(a_2, a_3, a_4)$ (since the quadruple mixed points on $a_2$ must be $(a_2, a_5), \;(a_2, a_6)$ and $(a_2, a_1)$) and for the $b$-line subarrangements.

Obviously, $(a_5,  a _6)$ cannot make a triple point with a $b$-type line, since this would imply, by \eqref{ndiv}, $a=b$.

 Assume $(a_5,  a _6)$ makes a triple point with another line of type $a$; one may assume without loosing generality that $(a_5,  a_6, a_2)$ is this triple point. In this context, let us examine the other multiple points on $a_2$. The point at the intersection of $a_2$ to $a_3$ must be a quadruple point, with two additional lines of type $b$. However, there are no possible choices among $ b_i, i= \overline {1, 6}$ for the lines of type $b$ to pass through the intersection of $a_2$ to $a_3$.  In conclusion, there are no triple points of type $(a_5, a _6, 
 a_i)$, with $i=2, 3, 4$.
 
 We are left with excluding the case when $(a_5, a _6)$  is a quadruple point. The other two lines in this quadruple point must be of type $b$. Assume, without losing generality, that $(a_5, a _6, b_1, b_3)$ is the quadruple  point. As $a_6$ contains three quadruple points, there are two remaining quadruple points to outline. 
 
 Assume, for instance, that $(a_6, a _2)$ and $(a_6, a _3)$ are the remaining quadruple points. Then one necessarily gets the quadruple points  $(a_6, a _2, b_4, b_5)$ (or $b_6$ instead of $b_5$, but this is a symmetric case) and  $(a_6, a _3, b_2, b_6)$. It follows that $(a_5, a_2)$ is a double point, and so is $(a_5, a_3)$. But this means that $a_5$ contains three double points (since $(a_1, a_5)$ was also a double point), contradiction.
 
 The only other distinct possibility (discarding the symmetries) is for $( a_6, a _3)$ and $(a_6, a _4)$ to form the remaining quadruple points $(a_6,  a _3, b_2, b_5)$ and $(a_6, a _4, b_4, b_6)$ on $a_6$. Similarly this leads to the conclusion that the line $a_5$ has three double points, contradiction.

The last claim follows from the fact that for a line arrangement being a net is a combinatorial property, see for instance \cite{FY} or \cite{Su13}.

Finally, we show that an arrangement $\A$ with $14$ lines cannot have non-trivial monodromy.  Assuming the contrary, for $|\A|=14$, would imply  $b_2(\A)>0$.

Hence, as before, the system $\bf S$ with $\Z_2$ coefficients would admit a non-constant solution $(a_H)_{H \in \A}$. This defines a partition of  $\A$ into proper subsets $\LL \sqcup \Ka$ such that $a_H=a$, for all $H \in \LL$ and $a_H=b \neq a$, for all $H \in \Ka$. This is only possible when each line in $\LL$ intersects each line in $\Ka$ in a quadratic point of type $(LKLK)$. 

Counting the quadruple points of type $(LKLK)$ on an arbitrary line $K\in \Ka$, we get that $|\LL|=2l$ and $3l+1 \leq 14$, so $|\LL| \in \{2,4,6,8\}$. There are actually only three distinct cases, $|\LL| \in \{2,4,6\}$, since $|\LL|=8 \Leftrightarrow |\Ka|=6$ and we are back to the third case.

The first two cases are easily dismissed. Assume $|\LL|\in \{2, 4\}$, and consider a line $L\in \LL$. Since $|\Ka| \geq 10$, there must be multiple points  on $L$ of type $(LK), \; (LKK),$ \linebreak
$(LKKK)$ or $(LKKKK)$. In any case, by Lemma \ref{beta}, we get $a=b$, contradiction.
In the last case, $|L|=6$, we necessarily have $4$ quadruple points of type $(LKLK)$ on each line in $\LL$ and  $3$ quadruple points of type $(LKLK)$ on each line in $\Ka$, otherwise by Lemma \ref{beta} we would get $a=b$.  It follows from Lemma \ref{14partition} that such an arrangement does not exist.

\end{proof}

 \begin{remark}
 \label{conj}
Libgober's result discussed above, Theorem \ref{withcvaduple} and all the examples we know so far suggest that the following property {\bf (P) } holds for hyperplane arrangement complements.

\bigskip

 {\bf (P) }  An equimonodromical rank one local system $\rho(m)$ for $m$ dividing $|\A|$ belongs to the characteristic variety
$$V^1(M)=\{\rho \in \T(M) ~~~~ | ~~~~H^1(M, _{\rho}\C) \ne 0 \}$$
if and only if there is a strictly positive dimensional irreducible component $W_m$ of $V^1(M)$ passing through the origin $1 \in \T(M)$ and such that $\rho(m) \in W_m$. 

\bigskip

This remark follows from the well known correspondence between the irreducible components of the characteristic variety $V^1(M)$ passing through the origin and the pencils on $M$, see for instance \cite{D3} or \cite{Su13}.
In most of the examples we know, one has in addition
$$\dim H^1(M, _{\rho(m)}\C)=\dim H^1(M, _{\rho}\C)=\dim W_m-1,$$
for $\rho \in W_m $ generic. Such an equality implies that the component $W_m$ is unique in view of Proposition 6.9 in \cite{ACM}. 

This equality fails however for the Ceva pencil described in Theorem \ref{3nets} $(1)$. Using the description of the corresponding resonance variety given in Example 2.14 in \cite{Su13}, we see that in this case there are $4$ irreducible components of dimension two passing through  the character $\rho(3)$ (and through its conjugate). See also Example 5.9 in loc.cit.
 \end{remark}

\bibliographystyle{amsalpha}

\end{document}